\documentclass[journal,twoside,web]{ieeecolor}
\usepackage{generic}
\usepackage{cite}

\usepackage{graphicx,algorithm,algpseudocode,amsthm,amsmath,amsfonts,verbatim,mathtools,enumerate,bm,hyperref,multirow,makecell,adjustbox}
\usepackage{siunitx}
\usepackage[style=base]{caption}
\usepackage{subcaption}

\usepackage[dvipsnames]{xcolor}
\usepackage{tikz,circuitikz}
\usetikzlibrary{patterns}
\usetikzlibrary{decorations.pathreplacing,positioning}
\usetikzlibrary{calc}

\allowdisplaybreaks

\newcommand{\rec}{\mathrm{rec\,}}

\newcommand{\argmin}{\mathop{\rm argmin}}

\newcommand{\norm}[1]{\left\lVert#1\right\rVert}
\newcommand{\mnorm}[1]{{\left\vert\kern-0.25ex\left\vert\kern-0.25ex\left\vert #1 
    \right\vert\kern-0.25ex\right\vert\kern-0.25ex\right\vert}}

\newtheorem{theorem}{Theorem}
\newtheorem{lemma}{Lemma}

\newtheorem{remark}{Remark}

\newtheorem{assumption}{Assumption}

\newcommand{\ie}{{\it i.e.}}

\usepackage{textcomp}
\def\BibTeX{{\rm B\kern-.05em{\sc i\kern-.025em b}\kern-.08em
    T\kern-.1667em\lower.7ex\hbox{E}\kern-.125emX}}

\begin{document}

\title{Extrapolated Proportional-Integral Projected Gradient Method for Conic Optimization}
\author{Yue~Yu, Purnanand~Elango, Beh\c{c}et~A\c{c}\i kme\c{s}e, and Ufuk Topcu
\thanks{Y. Yu and U. Topcu are with the Oden Institude of Computational Sciences and Engineering, The University of Texas at Austin, TX, 78712, USA (emails:  yueyu@utexas.edu,\,utopcu@utexas.edu). P. Elango and B. A\c{c}\i kme\c{s}e are with the Department of Aeronautics and Astronautics, University of Washington, Seattle, WA, 98195, USA (emails: pelango@uw.edu,\,behcet@uw.edu).}}

\maketitle
\thispagestyle{empty}
\pagestyle{empty}

\begin{abstract}
Conic optimization is the minimization of a convex quadratic function subject to conic constraints. We introduce a novel first-order method for conic optimization, named \emph{extrapolated proportional-integral projected gradient method (xPIPG)}, that automatically detects infeasibility. The iterates of xPIPG either asymptotically satisfy a set of primal-dual optimality conditions, or generate a proof of primal or dual infeasibility. We demonstrate the application of xPIPG using benchmark problems in model predictive control. xPIPG outperforms many state-of-the-art conic optimization solvers, especially when solving large-scale problems.
\end{abstract}



\section{Introduction}
\label{sec: introduction}
\IEEEPARstart{C}{onic} optimization is the minimization of a quadratic function subject to conic constraints:
\begin{equation}\label{opt: conic}
\begin{array}{ll}
\underset{z}{\mbox{minimize}} & \frac{1}{2}z^\top P z+q^\top z\\
\mbox{subject to} & Hz-g\in\mathbb{K}, \enskip z\in\mathbb{D},
\end{array}
\end{equation}
where \(P\in\mathbb{R}^{n\times n}\) is a symmetric positive semidefinite matrix, \(H\in\mathbb{R}^{m\times n}\), \(g\in\mathbb{R}^m\), \(\mathbb{K}\subset\mathbb{R}^m\) is a closed convex cone, and \(\mathbb{D}\subset\mathbb{R}^n\) is a closed convex set.

An important subproblem in conic optimization is to detect whether optimization~\eqref{opt: conic} is primal or dual infeasible  \cite{o2016conic,banjac2019infeasibility,banjac2021asymptotic,o2021operator,yu2021infeasibility,applegate2021infeasibility}. Optimization \eqref{opt: conic} is \emph{primal infeasible} if there exists \(\overline{w}\in\mathbb{R}^m\) such that
\begin{equation}\label{eqn: primal inf}
\underset{z\in\mathbb{D}}{\inf} \,\langle Hz-g, \overline{w}\rangle>\underset{y\in\mathbb{K}}{\sup}\,\langle y, \overline{w}\rangle.
\end{equation}
Here the vector \(\overline{w}\) is known as a \emph{proof of primal infeasibility}; it defines a hyperplane separating the set \(\{Hz-g\,|\,z\in\mathbb{D}\}\) from the cone \(\mathbb{K}\). On the other hand, optimization \eqref{opt: conic} is \emph{dual infeasible} if there exists \(\overline{z}\in\mathbb{R}^n\) such that the following conditions hold for all \(z\in\mathbb{D}\): 
\begin{equation}\label{eqn: dual inf}
P\overline{z}=0, \enskip q^\top \overline{z}<0, \enskip H\overline{z}\in\mathbb{K}, \enskip z+\overline{z}\in\mathbb{D}.
\end{equation}
Here the vector \(\overline{z}\) is known as a \emph{proof of dual infeasibility}; it defines a direction along which the value of the objective function of optimization~\eqref{opt: conic} can decrease indefinitely. We refer the interested readers to \cite[Prop. 3.1]{banjac2019infeasibility} and \cite[Prop. 4.1]{banjac2021minimal} for details on the infeasibility conditions in \eqref{eqn: primal inf} and \eqref{eqn: dual inf}. Infeasibility detection is the problem of computing, if possible, a proof of primal infeasibility or dual infeasibility \cite{banjac2019infeasibility,banjac2021asymptotic,yu2021infeasibility}. Such computation is necessary in both quasiconvex and mixed-integer optimization \cite{yu2021infeasibility}.

Traditional conic optimization methods detect infeasibility by computing matrix inverse (or equivalently, solving system of linear equations), usually as a subroutine of the interior-point method \cite{nesterov1999infeasible} or the Douglas-Rachford-splitting method \cite{raghunathan2014infeasibility,o2016conic,liu2019new,ryu2019douglas,banjac2019infeasibility,o2021operator}. Such computation is numerically expensive for large-scale problems.

Proportional-integral projected gradient method (PIPG) is the first method that detects infeasibility in general conic optimization without solving systems of linear equations \footnote{Another method introduced in \cite{applegate2021infeasibility} also avoid solving systems of linear equations. But this method only applies to linear programs.}  \cite{yu2021infeasibility}. If the conic optimization is feasible, PIPG also enjoys the best convergence rates among first-order primal-dual methods. For a detailed comparison between PIPG and other first-order methods, see \cite{chambolle2016ergodic,yu2020proportional,yu2021proportional}.

By adding an extrapolation step to PIPG, we introduce a novel conic optimization method, named \emph{extrapolated PIPG} (xPIPG), that automatically detects infeasibility. Previous results only showed the ergodic convergence of the primal-dual gap function for a variant of xPIPG  \cite{chambolle2016ergodic}. We prove that the iterates of xPIPG either asymptotically satisfy a set of primal-dual optimality conditions, or generates a proof of primal or dual infeasibility.

We demonstrate the application of xPIPG using the optimal control problem of oscillating masses, a popular benchmark problem in model predictive control \cite{wang2009fast,jerez2014embedded,yu2021proportional}. Empirically, xPIPG is about twice as fast as PIPG. With an efficient implementation in C language, xPIPG outperforms many state-of-the-art solvers, especially when solving large-scale problems. 


\section{Notation and preliminaries}
\label{sec: preliminary}

We let \(\mathbb{N}\), \(\mathbb{R}\) and \(\mathbb{R}_+\) denote the set of positive integers, real numbers, and non-negative real numbers, respectively. For two vectors \(z, z'\in\mathbb{R}^n\) and a positive definite matrix \(M\), we let \(\langle z, z'\rangle\) denote the inner product of \(z\) and \(z'\), and \(\langle z, z'\rangle_M\) denote the inner product of \(z\) and \(z'\) weighted by matrix \(M\); we let \(\norm{z}\coloneqq \sqrt{\langle z, z\rangle}\) denote the \(\ell_2\)-norm of \(z\), and \(\norm{z}_M\coloneqq \sqrt{\langle z, Mz\rangle}\) denote the norm of \(z\) weighted by matrix \(M\). We let \(1_n\) denote the \(n\)-dimensional vector of all 1's, \(0_n\) denote the \(n\)-dimensional zero vector, and \(I_n\) denote the \(n\times n\) identity matrix. For a matrix \(H\in\mathbb{R}^{m\times n}\), we let \(H^\top\) denote its transpose, \(\norm{H}\) denote its largest singular value. A set \(\mathbb{D}\subseteq\mathbb{R}^n\) is convex if \(\alpha z+(1-\alpha)z'\in\mathbb{D}\) for any \(\alpha\in[0, 1]\) and \(z, z'\in\mathbb{D}\). A set \(\mathbb{K}\subseteq\mathbb{R}^m\) is a convex cone if \(\mathbb{K}\) is convex and \(\alpha w\in\mathbb{K}\) for any \(w\in\mathbb{K}\) and \(\alpha\in\mathbb{R}_+\).

 Let \(\mathbb{D}\subseteq\mathbb{R}^n\) be a nonempty closed convex set. 
 The \emph{point-to-set distance} from \(z\in\mathbb{R}^n\) to set \(\mathbb{D}\) is given by
\begin{equation}\label{eqn: distance func}
    d(z\,|\,\mathbb{D})\coloneqq \underset{z'\in\mathbb{D}}{\min} \norm{z-z'}.
\end{equation}
 The \emph{projection} of \(z\) onto set \(\mathbb{D}\) is given by
\begin{equation}\label{eqn: projection}
\pi_{\mathbb{D}}[z]\coloneqq \underset{z'\in\mathbb{D}}{\argmin} \norm{z-z'}.
\end{equation}
The \emph{normal cone} of set \(\mathbb{D}\) at \(z\) is given by
\begin{equation}\label{eqn: normal cone}
    N_{\mathbb{D}}(z)\coloneqq\{y\in\mathbb{R}^n\,|\,\langle y, z'-z\rangle\leq 0, \forall z'\in\mathbb{D}\}.
\end{equation}
The \emph{recession cone} of set \(\mathbb{D}\) is a given by
\begin{equation}\label{eqn: rec cone}
    \rec \mathbb{D}\coloneqq \{y\in\mathbb{R}^n\,|\, y+z\in\mathbb{D}, \forall\, z\in\mathbb{D}\}.
\end{equation}
The \emph{support function} of set \(\mathbb{D}\) is given by
\begin{equation}\label{eqn: support}
    \sigma_{\mathbb{D}}(z)\coloneqq \sup_{y\in\mathbb{D}}\,\langle y, z\rangle
\end{equation}
for all \(z\in\mathbb{R}^n\). Let \(\mathbb{K}\subseteq\mathbb{R}^m\) be a nonempty closed convex cone. The recession cone of cone \(\mathbb{K}\) is itself, \ie, \(\rec \mathbb{K}=\mathbb{K}\) \cite[Cor. 6.50]{bauschke2017convex}.
The \emph{polar cone} of \(\mathbb{K}\) is  given by
\begin{equation}\label{eqn: polar cone}
    \mathbb{K}^\circ\coloneqq \{w\in\mathbb{R}^m\,|\, \langle w, y\rangle\leq 0, \forall\, y\in\mathbb{K}\}.
\end{equation}
Let \(M\in\mathbb{R}^{p\times p}\) be a symmetric positive definite matrix. A function \(T:\mathbb{R}^{p}\to\mathbb{R}^p\) is a \emph{\(\gamma\)-averaged operator} for some \(\gamma\in(0, 1)\) with respect to norm \(\norm{\cdot}_M\) if and only if the following condition holds for any \(\zeta_1, \zeta_2\in\mathbb{R}^p\) \cite[Prop. 4.35]{bauschke2017convex}:
\begin{equation}\label{def: avg op}
\begin{aligned}
   &\norm{T(\zeta_1)-T(\zeta_2)}_M^2-\norm{\zeta_1-\zeta_2}_M^2\\
   &\textstyle \leq \frac{\gamma-1}{\gamma}\norm{\zeta_1-\zeta_2-T(\zeta_1)+T(\zeta_2)}_M^2. 
\end{aligned}
\end{equation}

We will use the following results on averaged operators.

\begin{lemma}\cite[Lem. 5.1]{banjac2019infeasibility}
\label{lem: min dis}
Let \(M\in\mathbb{R}^{p\times p}\) be a symmetric positive definite matrix and \(T:\mathbb{R}^p\to\mathbb{R}^p\) be a \(\gamma\)-averaged operator for some \(\gamma\in(0, 1)\) with respect to norm \(\norm{\cdot}_M\). Let \(\zeta^1\in\mathbb{R}^p\) and \(\zeta^{j+1}\coloneqq T(\zeta^j)\) for all \(j\in\mathbb{N}\). Then there exists \(\overline{\zeta}\in\mathbb{R}^p\), known as the \emph{minimal-displacement vector} of operator \(T\), such that \(\textstyle \lim\limits_{j\to\infty} \frac{\zeta^j}{j}=\lim\limits_{j\to\infty}\zeta^{j+1}-\zeta^j=\overline{\zeta}\).
\end{lemma}

\section{Extrapolated proportional-integral projected gradient method}
\label{sec: method}

We introduce our main contribution, a conic optimization method that automatically detects infeasibility, in Algorithm~\ref{alg: PIPG}, where \(\alpha, \beta, \rho\in\mathbb{R}_+\) are positive step sizes, \(k\in\mathbb{N}\) is the maximum number of iterations, \(\epsilon\in\mathbb{R}_+\) is a small tolerance for infeasibility. 

We name Algorithm~\ref{alg: PIPG} the \emph{extrapolated proportional-integral projected gradient method} for the following reasons. First, if \(\rho=1\), then Algorithm~\ref{alg: PIPG} reduces to the proportional-integral projected gradient method (PIPG) \cite{chambolle2016ergodic,yu2020proportional,yu2021proportional}, which has been used in distributed optimization \cite{yu2020mass,yu2020rlc} and optimal control \cite{yu2020proportional,elango2022customised}. Second, lines~\ref{alg: relaxed z}--\ref{alg: relaxed w} define an interpolation step if \(\rho\in[0, 1]\) and an extrapolation step if \(\rho\in[1, 2]\); see Fig.~\ref{fig: extrapolation} for an illustration. Since extrapolation improves practical convergence \cite{eckstein1992douglas,chambolle2016ergodic}, we will let \(\rho\in[1, 2]\) in Algorithm~\ref{alg: PIPG}, hence the name ``extrapolated PIPG".

\begin{algorithm}[!ht]
\caption{xPIPG}
\begin{algorithmic}[1]
\Require \(k, \alpha, \beta, \rho,  \epsilon\), initial values  \(\xi^1, \eta^1\).
\For{\(j=1, 2, \ldots, k\)}\label{alg: pipg for start}
\State{\(z^{j+1}=\pi_{\mathbb{D}}[\xi^j-\alpha( P\xi^j+q+H^\top \eta^j)]\)}\label{alg: z}
\State{\(w^{j+1}= \pi_{\mathbb{K}^\circ}[\eta^j+\beta(H(2z^{j+1}-\xi^j)-g)]\)}\label{alg: w}
\State{\(\xi^{j+1}=(1-\rho)\xi^j+\rho z^{j+1}\)}\label{alg: relaxed z}
\State{\(\eta^{j+1}=(1-\rho)\eta^j+\rho w^{j+1}\)}\label{alg: relaxed w}
\EndFor\label{alg: pipg for end}
\If{\(\max\{\frac{1}{\alpha\rho}\norm{z^{k+1}-z^k}, \frac{1}{\beta\rho}\norm{w^{k+1}-w^k}\}\leq \epsilon\)}\label{eqn: PIPG test start}
\State{\Return{\(z^k\)}} 
\Else 
\If{\(\frac{1}{\beta\rho}\norm{w^{k+1}-w^k} > \epsilon\)}
\State{\Return{``Primal infeasible''}}
\EndIf
\If{\(\frac{1}{\alpha\rho}\norm{z^{k+1}-z^k} > \epsilon\)}
\State{\Return{``Dual infeasible''}}
\EndIf
\EndIf\label{eqn: PIPG test end}
\end{algorithmic}
\label{alg: PIPG}
\end{algorithm}

\begin{figure}
    \centering
    \begin{tikzpicture}
    \coordinate (a) at (-4, 0);
    \coordinate (b) at (0, 0);
    \coordinate (c) at (4, 0);
    
    \draw[ultra thick] (a) -- (c);
    \fill[blue] (a) circle [radius=2pt];
    \node[label=below:{\footnotesize \color{blue} $\xi^k$ }] at (a) {};

    \fill[red] (b) circle [radius=2pt];
    \node[label=below:{\footnotesize \color{red} $z^{k+1}$ }] at (b) {};

    \fill[Green] (c) circle [radius=2pt];
    
    \draw[decorate,decoration={brace,amplitude=10pt}]
($(a)+(0, 0.1)$) -- ($(c)+(0, 0.1)$) node [midway,yshift=0.5cm,above] {\footnotesize The range of \(\xi^{k+1}\) when $\rho\in[0, 2]$};

    \draw[decorate,decoration={brace,amplitude=5pt,mirror}]
($(a)+(0, -0.1)$) -- ($(b)+(0, -0.1)$) node [midway,yshift=-0.3cm,below] {\footnotesize $\rho\in[0, 1]$, interpolation};
    \draw[decorate,decoration={brace,amplitude=5pt,mirror}]
($(b)+(0, -0.1)$) -- ($(c)+(0, -0.1)$) node [midway,yshift=-0.3cm,below] {\footnotesize $\rho\in[1, 2]$, extrapolation};

    \end{tikzpicture}
    \caption{The range of \(\xi^{j+1}\) in Algorithm~\ref{alg: PIPG}.}
    \label{fig: extrapolation}
\end{figure}
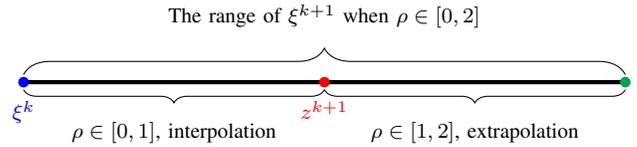

Previous results already showed that the iterations in Algorithm~\ref{alg: PIPG} ensure that, along certain averaged sequence of iterates, the primal-dual gap function converges to zero at the rate of \(O(1/k)\) \cite[Thm. 2]{chambolle2016ergodic}.
We will further show that for using a large enough iteration number \(k\) and a small enough tolerance \(\epsilon\),  Algorithm~\ref{alg: PIPG} automatically detects the primal and dual infeasibility of optimization~\eqref{opt: conic}, just like its special case where \(\rho=1\) \cite[Thm. 1]{yu2021proportional}.
 
We start with the following assumptions.

\begin{assumption}\label{asp: basic}
Matrix \(P\in\mathbb{R}^{n\times n}\) is symmetric and positive semidefinite, set \(\mathbb{D}\subset\mathbb{R}^n\) is closed and convex, cone \(\mathbb{K}\subset\mathbb{R}^m\) is closed and convex.
\end{assumption}

With these assumptions, the following lemma shows that line~\ref{alg: z} to line~\ref{alg: relaxed w} in Algorithm~\ref{alg: PIPG} are equivalent to the fixed-point iteration of an averaged operator.  

\begin{lemma}\label{lem: nonexpansive}
Suppose that Assumption~\ref{asp: basic} holds.
Let \(\zeta^j=\begin{bmatrix}
(\xi^j)^\top &
(\eta^j)^\top
\end{bmatrix}^\top\) for all \(j=1, 2, \ldots, k\) in Algorithm~\ref{alg: PIPG} and \(M=\begin{bsmallmatrix}
    \frac{1}{\alpha} I-P & -H^\top\\
    -H & \frac{1}{\beta}I
    \end{bsmallmatrix}\). If \(\alpha, \beta>0\), \(\alpha(\norm{P}+\beta\norm{H}^2)<1\), and \(\rho\in(0, 2)\), then line \ref{alg: z} to ~line\ref{alg: relaxed w} in Algorithm~\ref{alg: PIPG} is equivalent to \(\zeta^{j+1}=T(\zeta^j)\), where \(T:\mathbb{R}^{m+n}\to\mathbb{R}^{m+n}\) is a \(\frac{\rho}{2}\)-averaged operator with respect to the norm \(\norm{\cdot}_M\).
\end{lemma}

\begin{proof}
See Appendix~\ref{app: lem}.
\end{proof}

\begin{remark}\label{rem: alpha beta}
The conditions for \(\alpha\) and \(\beta\) in Lemma~\ref{lem: nonexpansive} are equivalent to the following conditions for some \(\omega>0\):
\begin{equation}
    \textstyle 0<\alpha<\frac{2}{\sqrt{\norm{P}^2+4\omega\norm{H}^2}+\norm{P}}, \enskip \beta=\omega\alpha.
\end{equation}
\end{remark}

Equipped with Lemma~\ref{lem: min dis} and Lemma~\ref{lem: nonexpansive}, we are now ready to present our main theoretical results. The following theorem shows that, for a large enough \(k\), the iteration between line~\ref{alg: pipg for start} and line~\ref{alg: pipg for end} in Algorithm~\ref{alg: PIPG} ensures either a primal-dual optimality condition, a primal infeasibility condition, or a dual infeasibility condition.

\begin{theorem}\label{thm: inf detect}
Suppose that Assumption~\ref{asp: basic} holds and the sequence \(\{z^j, w^j\}_{j\in\mathbb{N}}\) is computed recursively using line \ref{alg: z} to line~\ref{alg: relaxed w} in Algorithm~\ref{alg: PIPG} with \(\alpha, \beta>0\), \(\alpha(\norm{P}+\beta\norm{H}^2)<1\), and \(\rho\in(0, 2)\). There exists \(\overline{z}\in\rec \mathbb{D}\) and \(\overline{w}\in\mathbb{K}^\circ\) such that 
\begin{equation}\label{eqn: limit points}
    \lim\limits_{j\to\infty}z^j-z^{j-1}=\overline{z}, \enskip \lim\limits_{j\to\infty}w^j-w^{j-1}=\overline{w},
\end{equation}
and the following conditions hold:
\begin{subequations}\label{eqn: inf detect}
\begin{align}
    &\inf_{z\in\mathbb{D}}\,\langle Hz-g, \overline{w}\rangle=\textstyle\underset{y\in\mathbb{K}}{\sup}\,\langle y, \overline{w} \rangle+\frac{1}{\rho\beta}\norm{\overline{w}}^2,  \label{eqn: separating}\\
    & H\overline{z}\in\mathbb{K}, \enskip P\overline{z}=0, \enskip \langle q, \overline
    z\rangle=\textstyle-\frac{1}{\rho\alpha}\norm{\overline{z}}^2.\label{eqn: improving}
\end{align}
\end{subequations}
Furthermore, if \(\norm{\overline{w}}=\norm{\overline{z}}=0\), then
\begin{subequations}\label{eqn: distance limit}
\begin{align}
&\lim_{j\to\infty} d(Hz^j-g| N_{\mathbb{K}^\circ}(w^j))=0,\\
        &\lim_{j\to\infty} d(-Pz^j-q-H^\top w^j|N_{\mathbb{D}}(z^j))=0.
\end{align}    
\end{subequations}
\end{theorem}

\begin{proof}
See Appendix~\ref{app: thm}.
\end{proof}

The two limits in \eqref{eqn: distance limit} imply that, as \(j\to\infty\), the following conditions hold: 
\begin{equation}\label{eqn: optimality}
Hz^j-g\in N_{\mathbb{K}^\circ}(w^j),\enskip -Pz^j-q-H^\top w^j\in N_{\mathbb{D}}(z^j).
\end{equation}
Assuming certain constraint qualification condition holds at \(z^j\), the conditions in \eqref{eqn: optimality} state that \(z^j\) satisfies the optimality condition of optimization~\eqref{opt: conic}; see \cite[Ex. 10.8, Ex. 11.46]{rockafellar2009variational} for details. Therefore, Theorem~\ref{thm: inf detect} shows that the iterates in Algorithm~\ref{alg: PIPG} either asymptotically satisfy the set of primal-dual optimality conditions in \eqref{eqn: distance limit}, or provide a proof of primal or dual infeasibility. In particular, if \(\norm{\overline{w}}\neq 0\), then \eqref{eqn: separating} implies the primal infeasibility condition in \eqref{eqn: primal inf}; if \(\norm{\overline{z}}\neq 0\), then \eqref{eqn: improving} implies the dual infeasibility conditions in \eqref{eqn: dual inf}.

\section{Numerical examples}
\label{sec: numerical}

We demonstrate the application of xPIPG in optimal control problems and compare its performance against state-of-the-art open-source and commercial solvers. To this end, we consider the two-point boundary-value optimal control problem of a mechanical system composed of \(l\in\mathbb{N}\) oscillating masses \cite{wang2009fast,jerez2014embedded,yu2021proportional}; see Fig.~\ref{fig: mass-spring model} for an illustration. 
Given an initial state \(\hat{x}_0\in\mathbb{R}^{2l}\) and final time index \(\tau\in\mathbb{N}\), the discrete-time optimal control problem of the oscillating masses system is as follows: 
\begin{equation}\label{opt: opt control}
    \begin{array}{ll}
        \underset{u_{[0, \tau-1]}, x_{[0, \tau]}}{\mbox{minimize}}  & \textstyle \frac{1}{2}\sum_{t=0}^{\tau} x_t^\top Qx_t+\frac{1}{2}\sum_{t=0}^{\tau-1}u_t^\top Ru_t\\
        \mbox{subject to}  &  x_{t+1}=Ax_t+Bu_t,\enskip0\leq t\leq \tau-1,\\
        &u_t\in\mathbb{U},\enskip 0\leq t\leq \tau-1,\\
        & x_{t}\in\mathbb{X},\enskip 1\leq t\leq \tau-1,\\
        & x_0=\hat{x}_0,\enskip x_{\tau}= 0_{2l},
    \end{array}
\end{equation}where \(u_{[0, \tau-1]}\coloneqq \begin{bmatrix}
 u_0^\top & u_1^\top & \cdots & u_{\tau-1}^\top 
    \end{bmatrix}^\top\) and \(x_{[0, \tau]}\coloneqq \begin{bmatrix}
    x_0^\top & x_1^\top & \cdots & x_{\tau}^\top\end{bmatrix}^\top\) are the input and state trajectory respectively. In addition, we let
\begin{equation*}
\begin{aligned}
    &\textstyle A=\exp\left(\Delta\begin{bsmallmatrix}
    0_{l\times l} & I_l\\
    -L & 0_{l\times l}
    \end{bsmallmatrix}\right), \enskip Q=I_{2l}, \enskip R=I_l,\\
    & B=\textstyle\int_{0}^\Delta \exp\left(s\begin{bsmallmatrix}
    0_{l\times l} & I_l\\
    -L & 0_{l\times l}
    \end{bsmallmatrix}\right)\mathrm{d}s \begin{bsmallmatrix}
    0_{l\times l}\\
    I_l
    \end{bsmallmatrix},\\
    &\textstyle \mathbb{X}=\{x\in\mathbb{R}^{2l}| \norm{x}_\infty \leq \varrho_x\},\enskip  \mathbb{U}=\{u\in\mathbb{R}^l| \norm{u}_\infty\leq \varrho_u\}, 
\end{aligned}
\end{equation*} 
where \(\Delta=0.1\) is the sampling period of the dynamics, \(L\in\mathbb{R}^{l\times l}\) is a tridiagonal matrix: its diagonal entries are \(2\)'s, its subdiagonal ans superdiagonal entries are \(-1\)'s.  Optimization~\eqref{opt: opt control} is an instance of optimization~\eqref{opt: conic} where\begin{equation}\label{eqn: param}
\begin{aligned}
    & z = \begin{bmatrix}
    x_{[0,\tau]}^\top & u_{[0, \tau-1]}^\top 
    \end{bmatrix}^\top,\enskip P=I_{3\tau l+2l}, \enskip q=0_{3\tau l+2l}, \\ 
    &H =\begin{bmatrix} \begin{bmatrix}
    0_{2l\tau\times 2l} & I_{2l\tau}
    \end{bmatrix}-\begin{bmatrix}
    I_{\tau}\otimes A & 0_{2l\tau\times 2l}
    \end{bmatrix} & -I_{\tau}\otimes B \end{bmatrix},\\
     & g = 0_{2 l\tau},\enskip \mathbb{K}=\{0_{2\tau l}\},\enskip \mathbb{D}=\{\hat{x}_0\}\times \mathbb{X}^{\tau-1}\times \{0_{2l}\}\times \mathbb{U}^\tau.
\end{aligned}    
\end{equation}
Here \(\otimes\) denotes the Kronecker product, \(\mathbb{X}^{\tau-1}\) and \(\mathbb{U}^\tau\) denote the Cartesian product of \(\tau-1\) copies of \(\mathbb{X}\) and \(\tau\) copies of \(\mathbb{U}\), respectively.

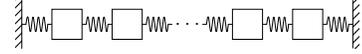
\begin{figure}[!ht]
\centering
\begin{adjustbox}{scale=0.4}

\begin{circuitikz}

\pattern[pattern=north east lines] (-0.2, -0.8) rectangle (0, 0.8);
\draw[thick] (0, -0.8) -- (0, 0.8);
\draw (0, 0) to[spring] (1, 0);
\draw (1, -0.5) rectangle (2, 0.5);
\draw (2, 0) to[spring] (3, 0);
\draw (3, -0.5) rectangle (4, 0.5);
\draw (4, 0) to[spring] (5, 0);

\node at (5.5, 0) {\huge $\ldots$};

\draw (6, 0) to[spring] (7, 0);
\draw (7, -0.5) rectangle (8, 0.5);
\draw (8, 0) to[spring] (9, 0);
\draw (9, -0.5) rectangle (10, 0.5);
\draw (10, 0) to[spring] (11, 0);

\pattern[pattern=north east lines] (11, -0.8) rectangle (11.2, 0.8);
\draw[thick] (11, -0.8) -- (11, 0.8);
\end{circuitikz}

\end{adjustbox}
\caption{The oscillating masses system}
\label{fig: mass-spring model}
\end{figure}

We apply Algorithm~\ref{alg: PIPG} (equipped with a C implenation, see \url{https://github.com/UW-ACL/pipg-demo/tree/master/xPIPG}) to instances of optimization~\eqref{opt: opt control} where \(\varrho_x=1\), \(\varrho_u=0.5\), \(\tau=20\) and \(l\in\{16, 32, 64, 128\}\). Furthermore, we sample the value of  \(\hat{x}_0\) from a normal distribution with mean \(\begin{bmatrix} \gamma 1_l^\top & 0_l^\top \end{bmatrix}^\top\) and standard deviation \(0.05 I_{2l}\), where \(\gamma=0.1\) for feasible instances and \(\gamma=0.8\) for infeasible instances of optimization~\eqref{opt: opt control}, respectively. Fig.~\ref{fig: rho} shows the convergence of xPIPG applied to one random problem instance where \(l=32\). xPIPG converges about twice as fast as PIPG; the latter is equivalent to xPIPG with \(\rho=1\). Furthermore, any value within the interval \([1.5, 1.9]\) provides a good choices for \(\rho\). These observations agree with those made for the Douglas-Rachford splitting methods \cite{eckstein1994parallel}. 

We compare the performance of Algorithm~\ref{alg: PIPG} (with \(\rho=1.6\)) against those of other state-of-the-art solvers, including OSQP \cite{stellato2020osqp}, SCS \cite{o2016conic,o2021operator}, ECOS \cite{domahidi2013ecos}, and MOSEK \cite{aps2019mosek}. Let \(z_a=\begin{bmatrix} \hat{x}_0^\top & -\varrho_x1_{2l(\tau-1)}^\top  & 0_{2l}^\top & -\varrho_u1_{\tau}^\top\end{bmatrix}^\top\) and \(z_b =\begin{bmatrix} \hat{x}_0^\top & \varrho_x 1_{2l(\tau-1)}^\top  & 0_{2l}^\top & \varrho_u 1_{l\tau}^\top\end{bmatrix}^\top\). For an feasible problem, we terminate Algorithm~\ref{alg: PIPG} when the following conditions hold:
\begin{equation}\label{eqn: feasible termination}
\begin{aligned}
    & \norm{Hz^j-g}_\infty \leq \epsilon_{\text{fea}},\\ &\norm{(z_b-z^j)\odot \max\{0_n, -Pz^j-q-H^\top w^j\}}_\infty\leq \epsilon_{\text{fea}},\\
    & \norm{(z_a-z^j)\odot \min\{0_n, -Pz^j-q-H^\top w^j\}}_\infty \leq \epsilon_{\text{fea}},
\end{aligned}
\end{equation}
where \(\max\) and \(\min\) are evaluated elementwise, and \(\odot\) is the elementwise product. One can verify that, when \(\epsilon_{\text{fea}}=0\), the conditions in \eqref{eqn: feasible termination} imply the optimality conditions in \eqref{eqn: optimality}; see \cite[Eqn. (9)]{stellato2020osqp}. We note that the \(\ell_\infty\) norm have also been used in the literature to measure the violation of other optimality conditions \cite{stellato2020osqp,o2021operator}. For an infeasible problem, we terminate Algorithm~\ref{alg: PIPG} when the following condition holds:
\begin{equation}\label{eqn: infeasible termination}
    \underset{z\in\mathbb{D}}{\inf}\langle Hz-g, w^{j+1}-w^j\rangle+\epsilon_{\text{inf}}>0.
\end{equation}
Since \(\mathbb{K}=\{0_{2\tau l}\}\), \eqref{eqn: infeasible termination} implies that \eqref{eqn: primal inf} holds up to a tolerance of \(\epsilon_{\text{inf}}\); furthermore, the infimum in \eqref{eqn: infeasible termination} has a closed form solution since set \(\mathbb{D}\) is a box. For the other solvers, we set their corresponding feasibility tolerance to \(\epsilon_{\text{fea}}\) and their primal and dual optimality tolerance to \(\epsilon_{\text{inf}}\). 

Tab.~\ref{tab: time} shows the computation time of different solvers, where the best computation time in each row is highlighted. We can see that Algorithm~\ref{alg: PIPG} has a clear advantage against open-source solvers--including OSQP, SCS, and ECOS--especially for large-scale and infeasible problems. The overall performance of Algorithm~\ref{alg: PIPG} is comparable to the highly optimized commercial solver MOSEK.

\begin{figure}[!ht]
\centering
  \begin{subfigure}{0.45\columnwidth}
  \includegraphics[width=\textwidth]{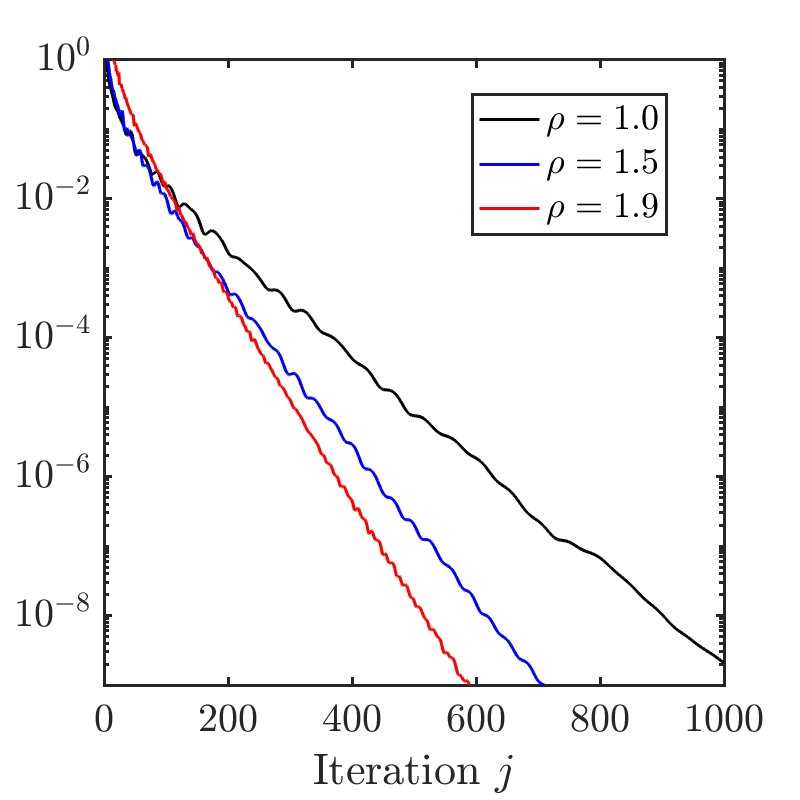}
  \caption{Feasible.} \label{fig: rho fea dual}
  \end{subfigure} 
  \begin{subfigure}{0.45\columnwidth}
  \includegraphics[width=\textwidth]{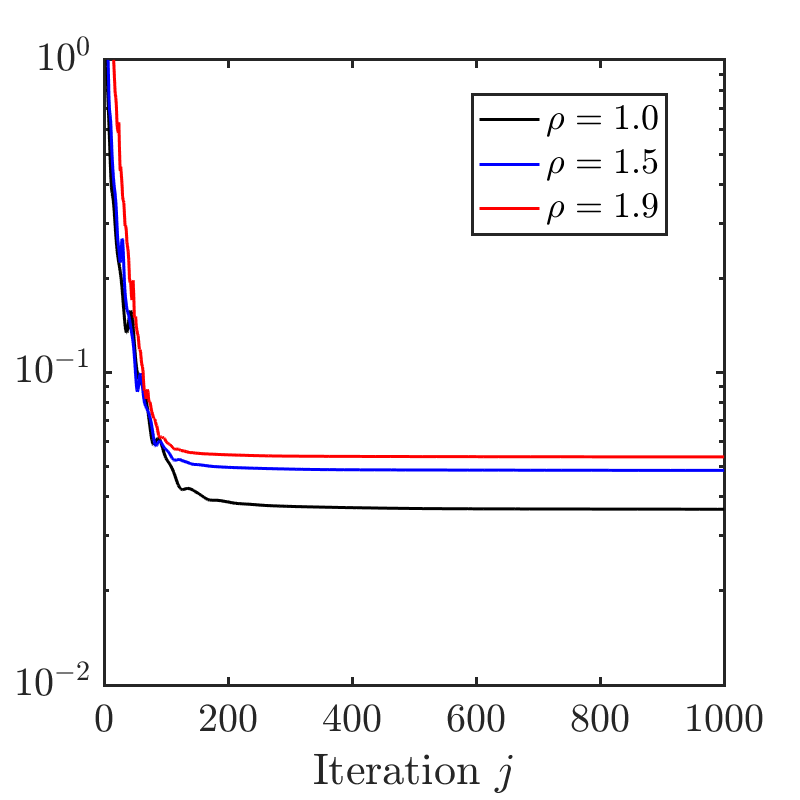}
  \caption{ Infeasible.} \label{fig: rho inf dual}
  \end{subfigure} 
  \caption{The asymptotic convergence of \(\frac{1}{\beta\rho}\norm{w^{j+1}-w^j}\) in Algorithm~\ref{alg: PIPG} when applied to an feasible (left) and infeasible (right) instance of optimization~\eqref{opt: opt control}. }
  \label{fig: rho}
\end{figure}

\begin{table}[!ht]
\caption{Comparison of the computation time (ms) of different solvers for instances of optimization~\eqref{opt: opt control} with randomly sampled values for \(\hat{x}_0\).}
\centering

\begin{subtable}[c]{0.48\textwidth}
    \centering
    \caption{Feasible problems, averaged over 100 random \(\hat{x}_0\).}
   \begin{tabular}{ c|c|ccccc } 
\hline
\makecell{\(\epsilon_{\text{fea}}\)} & \(l\) & OSQP & SCS & ECOS & MOSEK & xPIPG \\
\hline
\multirow{4}{2.5em}{\(10^{-4}\)}
& 16 & \textbf{12.9} & 15.5 & 63.7 & 26.8 & 15.7 \\
& 32 & 60.5 & 63.5 & 284.9 & 96.7 & \textbf{29.5} \\
& 64 & 437.7 & 449.3 & 2120.2 & 175.9 & \textbf{145.9}\\
& 128 & 2788.7 & 2846.9  & 20080.3 & \textbf{316.0} & 401.8 \\
\hline
\multirow{4}{2.5em}{\(10^{-8}\)} & 16 & \textbf{22.4} & 25.1 & 79.3 & 33.0 & 31.8 \\
& 32 & 112.1 & 102.9 & 363.3 & 117.6 & \textbf{61.8} \\
& 64 & 658.8 & 636.8 & 2448.6 & \textbf{215.2} & 269.7 \\
& 128 & 3655.3 & 3560.2 & 24860.3 & \textbf{374.6} & 771.8 \\
\hline
\end{tabular}

\end{subtable}
\\[0.2cm]
\begin{subtable}[c]{0.48\textwidth}
\centering
\caption{Infeasible problems, averaged over 100 random \(\hat{x}_0\).}

\begin{tabular}{ c|c|ccccc } 
\hline
\makecell{\(\epsilon_{\text{inf}}\)} & \(l\) & OSQP & SCS & ECOS & MOSEK & xPIPG \\
\hline
\multirow{4}{2.5em}{\(10^{-4}\)} & 16 & 21.9 & 26.1 & 47.4 & 26.2 & \textbf{8.6} \\
& 32 & 97.4 & 116.9 & 209.1 & 103.6 & \textbf{16.1} \\
& 64 & 799.2 & 845.4 & 1614.6 & 185.6 & \textbf{82.1} \\
& 128 & 4545.5 & 5563.9 & 13900.3 & 345.1 & \textbf{233.4} \\
\hline
\multirow{4}{2.5em}{\(10^{-8}\)} & 16 & 33.6 & 45.7 & 71.6 & 51.3 & \textbf{13.6} \\
& 32 & 92.3 & 118.7 & 208.3 & 122.7 & \textbf{16.0}\\
& 64 & 713.7 & 821.1 & 1432.0 & 216.1 & \textbf{56.7} \\
& 128 & 4367.7 & 5202.7 & 13611.8 & 285.9 & \textbf{211.7} \\
\hline
\end{tabular}

\end{subtable}
\label{tab: time}
\end{table}

\section{Conclusions}
\label{sec: conclusions}

We introduced a first-order conic optimization method, xPIPG, that automatically detects primal or dual infeasibility in conic optimization. With an efficient C implementation, xPIPG outperforms many state-of-the-art conic optimization solvers especially for large-scale problems. 

However, our software implementation and numerical experiments are still limited to a special class of optimal control problems. Our future work direction includes preconditioning and parameter selection in xPIPG for ill-conditioned problems, such as those in the Maros-Meszaros dataset \cite{stellato2020osqp,o2021operator}. 

\appendices

\section{Proof of Lemma~\ref{lem: nonexpansive}}
\label{app: lem}
Let \(\xi_i\in\mathbb{R}^n\), \(\eta_i\in\mathbb{R}^m\), and
\begin{subequations}
\begin{align}
    z_i^+=&\pi_{\mathbb{D}}[\xi_i-\alpha(P\xi_i+q+H^\top \eta_i)],\label{eqn: zi proj}\\
    w_i^+=&\pi_{\mathbb{K}^\circ}[\eta_i+\beta(H(2z_i^+-\xi_i)-g)],\label{eqn: wi proj}\\
    \xi_i^+=&(1-\rho)\xi_i+\rho z_i^+,\enskip 
    \eta_i^+=(1-\rho)\eta_i+\rho w_i^+,\label{eqn: zw relaxed}
\end{align}
\end{subequations}
for \(i=1, 2\). By applying \cite[Thm. 3.16]{bauschke2017convex} to the projections in \eqref{eqn: zi proj} and \eqref{eqn: wi proj} we can show the following:
\begin{subequations}
\begin{align}
    &\textstyle \langle \frac{1}{\alpha}(z_1^+-\xi_1)+P\xi_1+q+H^\top \eta_1, z_2^+-z_1^+\rangle\geq 0,\label{eqn: z1}\\
    &\textstyle \langle \frac{1}{\alpha}( z_2^+-\xi_2)+P\xi_2+q+H^\top \eta_2, z_1^+-z_2^+\rangle\geq 0,\label{eqn: z2}\\
    &\textstyle \langle \frac{1}{\beta}(w_1^+-\eta_1)-H(2z_1^+-\xi_1)+g, w_2^+-w_1^+ \rangle\geq 0,\label{eqn: w1}\\
    &\textstyle \langle \frac{1}{\beta}(w_2^+-\eta_2)-H(2z_2^+-\xi_2)+g, w_1^+-w_2^+ \rangle\geq 0,\label{eqn: w2}
\end{align}
\end{subequations}
Summing up \eqref{eqn: z1} and \eqref{eqn: z2} gives
\begin{equation}\label{eqn: z1 & z2}
    \begin{aligned}
    &0\leq \textstyle \frac{1}{\alpha}\langle z_1^+-z_2^+, \xi_1-\xi_2-z_1^++z_2^+\rangle\\
    &-\langle z_1^+-z_2^+, P(\xi_1-\xi_2) \rangle +\langle z_1^+-z_2^+, H^\top (\eta_2-\eta_1) \rangle. 
    \end{aligned}
\end{equation}
Similarly, summing up \eqref{eqn: w1} and \eqref{eqn: w2} gives
\begin{equation}\label{eqn: w1 & w2}
    \begin{aligned}
    &\textstyle 0\leq \frac{1}{\beta}\langle w_1^+-w_2^+, \eta_1-\eta_2-w_1^++w_2^+\rangle\\
    &-\langle w_1^+-w_2^+, H (\xi_1-\xi_2-2z_1^++2z_2^+) \rangle.
    \end{aligned}
\end{equation}
Since matrix \(P\) is positive semidefinite, we must have \(\langle z_1^+-z_2^+, P(z_1^+-z_2^+)\rangle\geq 0\). Using this fact we can show the following
\begin{equation}\label{eqn: P}
\begin{aligned}
    &-\langle z_1^+-z_2^+, P(\xi_1-\xi_2)\rangle\\
    &\leq -\langle z_1^+-z_2^+, P(\xi_1-\xi_2-z_1^++z_2^+)\rangle.
\end{aligned}
\end{equation}
Let 
\begin{equation}
   \zeta_1 = \begin{bsmallmatrix}
    \xi_1\\
    \eta_1
    \end{bsmallmatrix}, \enskip \zeta_2 = \begin{bsmallmatrix}
    \xi_2\\
    \eta_2
    \end{bsmallmatrix}, \enskip y_1^+ = \begin{bsmallmatrix}
    z_1^+\\
    w_1^+
    \end{bsmallmatrix}, \enskip y_2^+ = \begin{bsmallmatrix}
    z_2^+\\
    w_2^+
    \end{bsmallmatrix}.
\end{equation}
Summing up \eqref{eqn: z1 & z2}, \eqref{eqn: w1 & w2}, and \eqref{eqn: P} gives
\begin{equation}\label{eqn: M inner}
    \begin{aligned}
        0\leq \langle y_1^+-y_2^+, \zeta_1-\zeta_2-y_1^++y_2^+ \rangle_M.
    \end{aligned}
\end{equation}
Next, equation \eqref{eqn: zw relaxed} imply that
\(\textstyle y_i^+=\zeta_i+\frac{1}{\rho}(\zeta_i^+-\zeta_i)\) for \(i=1, 2\). Substituting this relation into \eqref{eqn: M inner} gives
\begin{equation}
    \begin{aligned}
        &\textstyle 0\leq \frac{1}{\rho^2}\langle \zeta_1^+-\zeta_2^+, \zeta_1-\zeta_2-\zeta_1^++\zeta_2^+\rangle_M\\
        &\textstyle +\frac{1-\rho}{\rho^2}\langle \zeta_1-\zeta_2, \zeta_1^+-\zeta_2^+-\zeta_1+\zeta_2\rangle_M\\
        &=\textstyle \frac{1}{2\rho^2}\norm{\zeta_1-\zeta_2}_M^2-\frac{1}{2\rho^2}\norm{\zeta_1^+-\zeta_2^+}_M^2\\
        &\textstyle-\frac{1}{2\rho^2}\norm{\zeta_1-\zeta_2-\zeta_1^++\zeta_2^+}_M^2+\frac{1-\rho}{2\rho^2}\norm{\zeta_1^+-\zeta_2^+}_M^2\\
         &\textstyle- \frac{1-\rho}{2\rho^2}\norm{\zeta_1-\zeta_2}_M^2-\frac{1-\rho}{2\rho^2}\norm{\zeta_1-\zeta_2-\zeta_1^++\zeta_2^+}_M^2\\
        &=\textstyle \frac{1}{2\rho}\norm{\zeta_1-\zeta_2}_M^2- \frac{1}{2\rho}\norm{\zeta_1^+-\zeta_2^+}_M^2\\
        &+\textstyle \frac{(\rho-2)}{2\rho^2}\norm{\zeta_1-\zeta_2-\zeta_1^++\zeta_2^+}_M^2,
    \end{aligned}
\end{equation}
where the first step is due to completion of squares. Letting \(\rho=2\gamma\) for some \(\gamma\in(0, 1)\) in the inequality above gives \eqref{def: avg op}, which completes the proof.
\section{Proof of Theorem~\ref{thm: inf detect}}
\label{app: thm}

We will use the results in \cite[Lem. 3.2]{banjac2021asymptotic}. First, by combining Lemma~\ref{lem: nonexpansive} and Lemma~\ref{lem: min dis}, we can show that there exists \(\overline{z}\in\mathbb{R}^n\) and \(\overline{w}\in\mathbb{R}^m\) such that
\begin{subequations}\label{eqn: xi lim}
\begin{align}
    &\textstyle\lim\limits_{j\to\infty} \xi^{j}-\xi^{j-1}=\lim\limits_{j\to\infty} \frac{\xi^{j}}{j}=\overline{z},\\
    &\textstyle\lim\limits_{j\to\infty} \eta^{j}-\eta^{j-1}=\lim\limits_{j\to\infty} \frac{\eta^{j}}{j}=\overline{w}.
\end{align}
\end{subequations}
 Next, line~\ref{alg: relaxed z} and line~\ref{alg: relaxed w} in Algorithm~\ref{alg: PIPG} implies that \(\rho(z^j-\xi^j)=\xi^j-\xi^{j-1}\), \(z^j=\xi^{j-1}+\frac{1}{\rho}(\xi^j-\xi^{j-1})\), \(\rho(w^j-\eta^j)=\eta^j-\eta^{j-1}\), and \(w^j=\eta^{j-1}+\frac{1}{\rho}(\eta^j-\eta^{j-1})\). By combining these equalities with \eqref{eqn: xi lim}, we can show the following:
\begin{equation}\label{eqn: z lim}
\begin{aligned}
    &\textstyle\lim\limits_{j\to\infty} z^j-z^{j-1}=\lim\limits_{j\to\infty} \rho(z^j-\xi^{j-1})=\lim\limits_{j\to\infty}\frac{z^j}{j}=\overline{z},\\
    &\textstyle\lim\limits_{j\to\infty} w^j-w^{j-1}=\lim\limits_{j\to\infty} \rho(w^j-\eta^{j-1})=\lim\limits_{j\to\infty}\frac{w^j}{j}=\overline{w}.
\end{aligned}
\end{equation}

 The above equation established \eqref{eqn: limit points}. Our next step is to prove the conditions in \eqref{eqn: distance limit}. By applying 
\cite[Thm. 27.4]{rockafellar2015convex} to line~\ref{alg: z} and line~\ref{alg: w} in Algorithm~\ref{alg: PIPG},
we can show the following:\begin{subequations}\label{eqn: zw normal cone}
\begin{align}
    &\textstyle \frac{1}{\alpha}(\xi^j-z^{j+1})-P\xi^j-q-H^\top\eta^j\in N_{\mathbb{D}}(z^{j+1}),\label{eqn: z normal cone}\\
    &\textstyle\frac{1}{\beta}(\eta^j-w^{j+1})+H(2z^{j+1}-\xi^j)-g\in N_{\mathbb{K}^\circ}(w^{j+1}).\label{eqn: w normal cone}
\end{align}
\end{subequations}
 According to the definition in \eqref{eqn: distance func}, the distance from point \(-Pz^{j+1}-q-Hw^{j+1}\) to the set \(N_{\mathbb{D}}(z^{j+1})\) is no larger than the distance from \(-Pz^{j+1}-q-Hw^{j+1}\) to \(\frac{1}{\alpha}(\xi^j-z^{j+1})-P\xi^j-q-H^\top\eta^j\); the reason is because \(\frac{1}{\alpha}(\xi^j-z^{j+1})-P\xi^j-q-H^\top\eta^j\) is, according to \eqref{eqn: z normal cone}, in set \(N_{\mathbb{D}}(z^{j+1})\). In other words, the following inequality holds:
\begin{equation}\label{eqn: z dist bound}
    \begin{aligned}
    &d(-Pz^{j+1}-q-Hw^{j+1}|N_{\mathbb{D}}(z^{j+1}))\\
    &\textstyle\leq \norm{(\frac{1}{\alpha}I-P)(z^{j+1}-\xi^j)-H^\top(w^{j+1}-\eta^j)}.
    \end{aligned}
\end{equation}
Similarly, we can show
\begin{equation}\label{eqn: w dist bound}
    \begin{aligned}
    &d(Hz^{j+1}-g|N_{\mathbb{K}^\circ}(w^{j+1}))\\
    &\textstyle \leq \norm{\frac{1}{\beta}(w^{j+1}-\eta^j)-H(z^{j+1}-\xi^j)}.
    \end{aligned}
\end{equation}
By combining the limits in \eqref{eqn: z lim}, \eqref{eqn: z dist bound}, and \eqref{eqn: w dist bound}, we can prove the limits in \eqref{eqn: distance limit}, assuming \(\norm{\overline{z}}=\norm{\overline{w}}=0\).

 We now prove the conditions in \eqref{eqn: inf detect}. To this end, let 
\begin{equation}\label{eqn: zw hat}
\begin{aligned}
    &\hat{z}^j=\xi^{j-1}-\alpha(P\xi^{j-1}+q+H^\top \eta^{j-1}),\\
    &\hat{w}^j=\eta^{j-1}+\beta(H(2z^j-\xi^{j-1})-g).
\end{aligned}
\end{equation}
 Dividing the above two equalities by \(j\) then letting \(j\to\infty\) gives the following:
\begin{equation}\label{eqn: hat limit}
    \textstyle\lim\limits_{j\to\infty}\frac{\hat{z}^j}{j}=\overline{z}-\alpha(P\overline{z}+H^\top \overline{w}),\enskip \textstyle\lim\limits_{j\to\infty}\frac{\hat{w}^j}{j}=\overline{w}+\beta H\overline{z},
\end{equation}
 where we again used \eqref{eqn: z lim}. Since \(z^j=\pi_{\mathbb{D}}[\hat{z}^j]\),  \(w^j=\pi_{\mathbb{K}^\circ}[\hat{w}^j]\) (due to line \ref{alg: z} and line \ref{alg: w} in Algorithm~\ref{alg: PIPG}), by using the results in \cite[Lem. 3.2]{banjac2021asymptotic} we can show the following:
\begin{subequations}
\begin{align}
    &\hat{z}^j-z^j\in(\rec \mathbb{D})^\circ, \enskip \hat{w}^j-w^j\in \mathbb{K}, \label{eqn: proj diff}\\
    &\lim\limits_{j\to\infty}\textstyle\frac{z_j}{j}=\overline{z}=\pi_{\rec\mathbb{D}}[\overline{z}-\alpha (P\overline{z}+H^\top \overline{w})],\label{eqn: D rec}\\
    &\lim\limits_{j\to\infty}\textstyle\frac{w_j}{j}=\overline{w}=\pi_{\mathbb{K}^\circ}[\overline{w}+\beta H\overline{z}],\label{eqn: K polar rec}\\
    &\begin{aligned}
    \lim\limits_{j\to\infty}\textstyle\frac{\hat{z}^j-z^j}{j}&=\pi_{(\rec\mathbb{D})^\circ}[\overline{z}-\alpha (P\overline{z}+H^\top \overline{w})]\\
    &=-\alpha(P\overline{z}+H^\top \overline{w}),
    \end{aligned}\label{eqn: D rec polar}\\
    & \lim\limits_{j\to\infty}\textstyle\frac{\hat{w}^j-w^j}{j}=\pi_{\mathbb{K}}[\overline{w}+\beta H \overline{z}]=\beta H\overline{z},\label{eqn: K polar rec polar}\\
    &\lim\limits_{j\to\infty}\textstyle\frac{\langle z^j,\hat{z}^j-z^j\rangle}{j}=\sigma_{\mathbb{D}}[-\alpha(P\overline{z}+H^\top \overline{w})],\label{eqn: supp D}\\
    &\lim\limits_{j\to\infty}\textstyle\frac{\langle w^j,\hat{w}^j-w^j\rangle}{j}=\sigma_{\mathbb{K}^\circ}[\beta H\overline{z}]=0,\label{eqn: supp K polar}
\end{align}
\end{subequations}
where we used the fact that \(\rec\mathbb{K}^\circ=\mathbb{K}^\circ\) and \((\mathbb{K}^\circ)^\circ=\mathbb{K}\); we also used \eqref{eqn: z lim} and \eqref{eqn: hat limit} in \eqref{eqn: D rec polar} and \eqref{eqn: K polar rec polar}. Notice that
\eqref{eqn: D rec}, \eqref{eqn: K polar rec}, and \eqref{eqn: K polar rec polar} directly implies the following
\begin{equation}\label{eqn: zw rec cone}
    \overline{z}\in\rec \mathbb{D},\enskip \overline{w}\in\mathbb{K}^\circ, \enskip H\overline{z}\in\mathbb{K}.
\end{equation}
By applying \cite[Thm. 6.30]{bauschke2017convex} to the projections in \eqref{eqn: D rec}, \eqref{eqn: D rec polar}, \eqref{eqn: K polar rec}, and \eqref{eqn: K polar rec polar} we can show the following
\begin{subequations}
\begin{align}
    &\langle \overline{z}, P\overline{z}+H^\top \overline{w}\rangle=0,\\
    &\langle \overline{w}, H\overline{z} \rangle=0.
    \label{eqn: complementary}
\end{align}
\end{subequations}
Combining the above two equalities with the assumption that \(P\) is positive semidefinite, we can conclude that
\begin{equation}\label{eqn: Pz=0}
    P\overline{z}=0.
\end{equation}
Next, we proceed to prove the conditions in \eqref{eqn: inf detect} (other than those in \eqref{eqn: zw rec cone} and \eqref{eqn: Pz=0}) using two argument. In the first argument, we start with the following:
\begin{subequations}
\begin{align}
    &\langle w^j, H\overline{z} \rangle\leq \langle -P\xi^{j-1}, \overline{z} \rangle,\label{eqn: bd sup 1}\\
    & \langle -Hz^j,  \overline{w}\rangle\leq \sigma_{\mathbb{D}}(-H^\top \overline{w}),\label{eqn: bd sup 2}
\end{align}
\end{subequations}
where \eqref{eqn: bd sup 1} is due to \eqref{eqn: zw rec cone}, \eqref{eqn: Pz=0}, and the definition of polar cone in \eqref{eqn: polar cone}; \eqref{eqn: bd sup 2} is due to \eqref{eqn: support}.

Furthermore, by using \eqref{eqn: zw hat}, \eqref{eqn: proj diff}, \eqref{eqn: D rec}, \eqref{eqn: K polar rec}, and the definition in \eqref{eqn: polar cone} we can show the following\begin{subequations}
\begin{align}
&\textstyle \langle \frac{1}{\alpha}(\xi^{j-1}-z^j)-P\xi^{j-1}-q-H^\top \eta^{j-1}, \overline{z}\rangle \leq 0,\label{eqn: bd sup 3}\\
&\textstyle \langle \frac{1}{\beta}(\eta^{j-1}-w^j)+H(2z^j-\xi^{j-1})-g, \overline{w} \rangle\leq 0.\label{eqn: bd sup 4}
\end{align}
\end{subequations}

By first summing up both sides of \eqref{eqn: bd sup 1} and \eqref{eqn: bd sup 3}, then letting \(j\to\infty\), we can show the following
\begin{equation}\label{eqn: two bounds a}
    \textstyle\frac{1}{\rho\alpha}\norm{\overline{z}}^2+\langle q, \overline{z}\rangle\geq \textstyle\frac{1}{\rho}\langle H\overline{z}, \overline{w}\rangle=0,
\end{equation}
where we also used \eqref{eqn: z lim} and\eqref{eqn: complementary}. Similarly, by first summing both sides of \eqref{eqn: bd sup 2} and \eqref{eqn: bd sup 4}, then letting \(j\to\infty\), we can show the following
\begin{equation}
\begin{aligned}
&\textstyle \frac{1}{\rho\beta}\norm{\overline{w}}^2+\langle g, \overline{w}\rangle+\sigma_{\mathbb{D}}(-H^\top \overline{w})\geq \textstyle\frac{1}{\rho}\langle H\overline{z}, \overline{w}\rangle=0,
\end{aligned}\label{eqn: two bounds b}
\end{equation}
where the last step is due to \eqref{eqn: complementary}.

Finally, by summing up both sides of \eqref{eqn: two bounds a} and \eqref{eqn: two bounds b}, we obtain the following intermediate step: 
\begin{equation}\label{eqn: inf eq 1}
    \textstyle\sigma_{\mathbb{D}}(-H^\top \overline{w})+\langle q, \overline{z}\rangle+\langle g, \overline{w}\rangle+\frac{1}{\rho\alpha}\norm{\overline{z}}^2+\frac{1}{\rho\beta}\norm{\overline{w}}^2\geq 0.
\end{equation}

Inequality \eqref{eqn: inf eq 1} completes our first argument. In the second argument, we will show that the inequality in \eqref{eqn: inf eq 1} actually holds as an equality, using the following three steps. First, combining \eqref{eqn: z lim} and \eqref{eqn: hat limit} gives the following
\begin{equation}
\begin{aligned}
    &\lim\limits_{j\to\infty}\textstyle\frac{1}{\alpha j}\langle z^j-\hat{z}^j, z^j-\xi^{j-1} \rangle=\frac{1}{\rho}\langle P\overline{z}+H^\top \overline{w}, \overline{z}\rangle=0,\\
    &\lim\limits_{j\to\infty}\textstyle\frac{1}{\beta j}\langle w^j-\hat{w}^j, w^j-\eta^{j-1} \rangle=\frac{1}{\rho}\langle H\overline{z}, \overline{w}\rangle=0.
\end{aligned}\label{eqn: inf eq 2}
\end{equation}
where the last step in the above two inequalities are due to \eqref{eqn: complementary} and \eqref{eqn: Pz=0}.  Second, by using \eqref{eqn: zw hat} we can show the following
\begin{equation}\label{eqn: inf eq 3}
    \begin{aligned}
       &\textstyle\frac{1}{\alpha j}\langle z^j-\hat{z}^j, \xi^{j-1}\rangle\\
       &\geq \textstyle \frac{1}{j}\langle \frac{1}{\alpha}( z^j-\xi^{j-1})+q+H^\top \eta^{j-1}, \xi^{j-1} \rangle,
    \end{aligned}
\end{equation}
where we used the assumption that \(P\) is positive semidefinite. Similarly, by using \eqref{eqn: zw hat} we can show the following
\begin{equation}\label{eqn: inf eq 4}
    \begin{aligned}
    &\textstyle \frac{1}{\beta j}\langle w^j-\hat{w}^j, \eta^{j-1} \rangle\\
    &=\textstyle \frac{1}{j}\langle \frac{1}{\beta}(w^j-\eta^{j-1})-H(2z^j-\xi^{j-1})+g, \eta^{j-1} \rangle.
    \end{aligned}
\end{equation}
Third, by first summing up the both sides of \eqref{eqn: inf eq 3}, and \eqref{eqn: inf eq 4}, then letting \(j\to\infty\), we obtain the following
\begin{equation}\label{eqn: sum two limits}
\begin{aligned}
&\textstyle \frac{1}{\alpha\rho}\norm{\overline{z}}^2+\frac{1}{\beta\rho}\norm{\overline{w}}^2+\langle q, \overline{z}\rangle+\langle g, \overline{w}\rangle-\frac{2}{\rho}\langle H\overline{z}, \overline{w}\rangle\\
&\leq \lim\limits_{j\to\infty}\textstyle\frac{1}{\alpha j}\langle z^j-\hat{z}^j, \xi^{j-1}\rangle+\lim\limits_{j\to\infty}\textstyle\frac{1}{\beta j}\langle w^j-\hat{w}^j, \eta^{j-1}\rangle,
\end{aligned}
\end{equation}
Notice the two limits in \eqref{eqn: sum two limits} always exist, due to \eqref{eqn: supp D}, \eqref{eqn: supp K polar}, and \eqref{eqn: inf eq 2}. By combining \eqref{eqn: sum two limits} together with \eqref{eqn: inf eq 2}, \eqref{eqn: supp D}, \eqref{eqn: supp K polar}, and \eqref{eqn: complementary} we obtain can the following
\begin{equation}\label{eqn: inf eq 6}
    \textstyle\sigma_{\mathbb{D}}(-H^\top \overline{w})+\langle q, \overline{z}\rangle+\langle g, \overline{w}\rangle+\frac{1}{\rho\alpha}\norm{\overline{z}}^2+\frac{1}{\rho\beta}\norm{\overline{w}}^2\leq 0,
\end{equation}
where we also used \eqref{eqn: Pz=0} and the fact that \(\sigma_{\mathbb{D}}(-\alpha H^\top \overline{w})=\alpha\sigma_{\mathbb{D}}(- H^\top \overline{w})\) for any \(\alpha\geq 0\). By combining \eqref{eqn: inf eq 1} and \eqref{eqn: inf eq 6} we conclude that inequality \eqref{eqn: inf eq 1} holds as an equality. Consequently, the inequalities in \eqref{eqn: two bounds a} and \eqref{eqn: two bounds b} must also hold as equalities, \ie,
\begin{equation}\label{eqn: two inf}
    \textstyle\langle q, \overline{z}\rangle=-\frac{1}{\rho\alpha}\norm{\overline{z}}^2,\enskip \textstyle \sigma_{\mathbb{D}}(-H^\top \overline{w})+\langle g, \overline{w}\rangle=-\frac{1}{\rho\beta}\norm{\overline{w}}^2.
\end{equation}
We now finished our second argument. Finally, by combining \eqref{eqn: zw rec cone}, the second equality in \eqref{eqn: two inf}, and the fact that \(\langle \overline{w}, y\rangle\leq 0\) for all \(y\in\mathbb{K}\)--which is due to \(\overline{w}\in\mathbb{K}^\circ\) in \eqref{eqn: zw rec cone}--we can obtain \eqref{eqn: separating}. By combining \eqref{eqn: zw rec cone}, \eqref{eqn: Pz=0}, and the first equality in \eqref{eqn: two inf}, we can obtain \eqref{eqn: improving}.



\bibliographystyle{IEEEtran}
\bibliography{IEEEabrv,reference}

\begin{thebibliography}{10}
\providecommand{\url}[1]{#1}
\csname url@rmstyle\endcsname
\providecommand{\newblock}{\relax}
\providecommand{\bibinfo}[2]{#2}
\providecommand\BIBentrySTDinterwordspacing{\spaceskip=0pt\relax}
\providecommand\BIBentryALTinterwordstretchfactor{4}
\providecommand\BIBentryALTinterwordspacing{\spaceskip=\fontdimen2\font plus
\BIBentryALTinterwordstretchfactor\fontdimen3\font minus
  \fontdimen4\font\relax}
\providecommand\BIBforeignlanguage[2]{{%
\expandafter\ifx\csname l@#1\endcsname\relax
\typeout{** WARNING: IEEEtran.bst: No hyphenation pattern has been}%
\typeout{** loaded for the language `#1'. Using the pattern for}%
\typeout{** the default language instead.}%
\else
\language=\csname l@#1\endcsname
\fi
#2}}

\bibitem{o2016conic}
B.~O’Donoghue, E.~Chu, N.~Parikh, and S.~Boyd, ``Conic optimization via
  operator splitting and homogeneous self-dual embedding,'' \emph{J. Optim.
  Theory Appl.}, vol. 169, no.~3, pp. 1042--1068, 2016.

\bibitem{banjac2019infeasibility}
G.~Banjac, P.~Goulart, B.~Stellato, and S.~Boyd, ``Infeasibility detection in
  the alternating direction method of multipliers for convex optimization,''
  \emph{J. Optim. Theory Appl.}, vol. 183, no.~2, pp. 490--519, 2019.

\bibitem{banjac2021asymptotic}
G.~Banjac and J.~Lygeros, ``On the asymptotic behavior of the
  {D}ouglas--{R}achford and proximal-point algorithms for convex
  optimization,'' \emph{Optim. Lett.}, pp. 1--14, 2021.

\bibitem{o2021operator}
B.~O'Donoghue, ``Operator splitting for a homogeneous embedding of the linear
  complementarity problem,'' \emph{SIAM J. Optim.}, vol.~31, no.~3, pp.
  1999--2023, 2021.

\bibitem{yu2021infeasibility}
Y.~Yu and U.~Topcu, ``Proportional-integral projected gradient method for
  infeasibility detection in conic optimization,'' \emph{arXiv preprint
  arXiv:2109.02756[math.OC]}, 2021.

\bibitem{applegate2021infeasibility}
D.~Applegate, M.~D{\'\i}az, H.~Lu, and M.~Lubin, ``Infeasibility detection with
  primal-dual hybrid gradient for large-scale linear programming,'' \emph{arXiv
  preprint arXiv:2102.04592[math.OC]}, 2021.

\bibitem{banjac2021minimal}
G.~Banjac, ``On the minimal displacement vector of the {D}ouglas--{R}achford
  operator,'' \emph{Oper. Res. Lett.}, vol.~49, no.~2, pp. 197--200, 2021.

\bibitem{nesterov1999infeasible}
Y.~Nesterov, M.~J. Todd, and Y.~Ye, ``Infeasible-start primal-dual methods and
  infeasibility detectors for nonlinear programming problems,'' \emph{Math.
  Program.}, vol.~84, no.~2, pp. 227--268, 1999.

\bibitem{raghunathan2014infeasibility}
A.~U. Raghunathan and S.~Di~Cairano, ``Infeasibility detection in alternating
  direction method of multipliers for convex quadratic programs,'' in
  \emph{Proc. IEEE Conf. Decision Control}.\hskip 1em plus 0.5em minus
  0.4em\relax IEEE, 2014, pp. 5819--5824.

\bibitem{liu2019new}
Y.~Liu, E.~K. Ryu, and W.~Yin, ``A new use of {D}ouglas--{R}achford splitting
  for identifying infeasible, unbounded, and pathological conic programs,''
  \emph{Math. Program.}, vol. 177, no.~1, pp. 225--253, 2019.

\bibitem{ryu2019douglas}
E.~K. Ryu, Y.~Liu, and W.~Yin, ``{D}ouglas--{R}achford splitting and {ADMM} for
  pathological convex optimization,'' \emph{Comput. Optim. Appl.}, vol.~74,
  no.~3, pp. 747--778, 2019.

\bibitem{chambolle2016ergodic}
A.~Chambolle and T.~Pock, ``On the ergodic convergence rates of a first-order
  primal--dual algorithm,'' \emph{Math. Program.}, vol. 159, no. 1-2, pp.
  253--287, 2016.

\bibitem{yu2020proportional}
Y.~Yu, P.~Elango, and B.~A{\c{c}}{\i}kme{\c{s}}e, ``Proportional-integral
  projected gradient method for model predictive control,'' \emph{IEEE Control
  Syst. Lett.}, vol.~5, no.~6, pp. 2174--2179, 2020.

\bibitem{yu2021proportional}
Y.~Yu, P.~Elango, U.~Topcu, and B.~A{\c{c}}{\i}kme{\c{s}}e,
  ``Proportional--integral projected gradient method for conic optimization,''
  \emph{Automatica}, vol. 142, p. 110359, 2022.

\bibitem{wang2009fast}
Y.~Wang and S.~Boyd, ``Fast model predictive control using online
  optimization,'' \emph{IEEE Trans. Control Syst. Technol.}, vol.~18, no.~2,
  pp. 267--278, 2009.

\bibitem{jerez2014embedded}
J.~L. Jerez, P.~J. Goulart, S.~Richter, G.~A. Constantinides, E.~C. Kerrigan,
  and M.~Morari, ``Embedded online optimization for model predictive control at
  megahertz rates,'' \emph{IEEE Trans. Autom. Control}, vol.~59, no.~12, pp.
  3238--3251, 2014.

\bibitem{bauschke2017convex}
H.~H. Bauschke and P.~L. Combettes, \emph{Convex {A}nalysis and {M}onotone
  {O}perator {T}heory in {H}ilbert {S}paces}.\hskip 1em plus 0.5em minus
  0.4em\relax Springer, 2017, vol. 408.

\bibitem{yu2020mass}
Y.~Yu, B.~A{\c{c}}{\i}kme{\c{s}}e, and M.~Mesbahi, ``Mass--spring--damper
  networks for distributed optimization in non-{E}uclidean spaces,''
  \emph{Automatica}, vol. 112, p. 108703, 2020.

\bibitem{yu2020rlc}
Y.~Yu and B.~A{\c{c}}{\i}kme{\c{s}}e, ``{RLC} circuits-based distributed mirror
  descent method,'' \emph{IEEE Control Syst. Lett.}, vol.~4, no.~3, pp.
  548--553, 2020.

\bibitem{elango2022customised}
P.~Elango, A.~Kamath, Y.~Yu, J.~M. Carson, and B.~Acikmese, ``A customised
  first-order solver for real-time powered-descent guidance,'' in \emph{Proc.
  AIAA Scitech Forum}, 2022, p. 0951.

\bibitem{eckstein1992douglas}
J.~Eckstein and D.~P. Bertsekas, ``On the {D}ouglas-{R}achford splitting method
  and the proximal point algorithm for maximal monotone operators,''
  \emph{Math. Program.}, vol.~55, no.~1, pp. 293--318, 1992.

\bibitem{rockafellar2009variational}
R.~T. Rockafellar and R.~J.-B. Wets, \emph{Variational analysis}.\hskip 1em
  plus 0.5em minus 0.4em\relax Springer Science \& Business Media, 2009, vol.
  317.

\bibitem{eckstein1994parallel}
J.~Eckstein, ``Parallel alternating direction multiplier decomposition of
  convex programs,'' \emph{J. Optim. Theory Appl.}, vol.~80, no.~1, pp. 39--62,
  1994.

\bibitem{stellato2020osqp}
B.~Stellato, G.~Banjac, P.~Goulart, A.~Bemporad, and S.~Boyd, ``{OSQP}: an
  operator splitting solver for quadratic programs,'' \emph{Math. Program.
  Comput.}, vol.~12, no.~4, pp. 637--672, 2020.

\bibitem{domahidi2013ecos}
A.~Domahidi, E.~Chu, and S.~Boyd, ``{ECOS}: An {SOCP} solver for embedded
  systems,'' in \emph{Proc. Eur. Control Conf.}\hskip 1em plus 0.5em minus
  0.4em\relax IEEE, 2013, pp. 3071--3076.

\bibitem{aps2019mosek}
{MOSEK ApS}, ``Mosek optimization toolbox for {MATLAB},'' \emph{User’s Guide
  and Reference Manual, Version}, vol.~4, 2019.

\bibitem{rockafellar2015convex}
R.~T. Rockafellar, \emph{Convex {A}nalysis}.\hskip 1em plus 0.5em minus
  0.4em\relax Princeton {U}niversity {P}ress, 2015.

\end{thebibliography}

\end{document}